\documentclass[reqno,12pt,draft]{amsart}
\usepackage{amsmath,amsthm,amssymb}

\date{}

\newtheorem{theorem}{Theorem}[section]
\newtheorem{lemma}[theorem]{Lemma}
\newtheorem{corollary}[theorem]{Corollary}
\newtheorem{proposition}[theorem]{Proposition}

\newtheorem{remark}[theorem]{Remark}
\newtheorem{definition}[theorem]{Definition}

\numberwithin{equation}{section}

\begin{document}

\centerline{\bf
Sobolev regularity for the  Monge--Amp{\`e}re
equation}

\centerline{\bf in the Wiener space}

\vskip .1in

\centerline{\bf Vladimir I.~Bogachev, Alexander V.~Kolesnikov}

\vskip .1in

\centerline{\bf Abstract}

\vskip .1cm

{\small
Given  the standard Gaussian measure $\gamma$
on the countable product of lines $\mathbb{R}^{\infty}$ and
a probability measure $g \cdot \gamma$
absolutely continuous with respect to $\gamma$, we consider
  the optimal transportation $T(x) = x + \nabla \varphi(x)$
of $g \cdot \gamma$  to $\gamma$. Assume that the function
$|\nabla g|^2/g$ is $\gamma$-integrable.
We prove that the function $\varphi$ is  regular
in a certain Sobolev-type sense and satisfies
the classical change of variables formula
$g = {\det}_2(\mbox{I} + D^2 \varphi) \exp \bigl( \mathcal{L} \varphi - \frac{1}{2} |\nabla \varphi|^2 \bigr)$.
We also establish sufficient conditions for the existence of
third order derivatives of  $\varphi$.
}

\vskip .1cm
%\maketitle



\section{Introduction}

 Numerous applications of the optimal transportation theory in finite-dimensional spaces
have been found during the last  decade. They include differential equations, probability theory, and geometry (see \cite{AGS}, \cite{Vill}).
The situation in infinite-dimensional spaces has been much less studied.
However, some  partial results on existence, uniqueness, and
regularity have been obtained in  \cite{FU1},  \cite{Kol04}, \cite{BoKo2005}, \cite{BoKo2006}.

In the finite-dimensional case any optimal transportation  mapping $T$ is a solution to
 the variational Monge--Kantorovich problem.
Assume we are given two probability measures $\mu$ and $\nu$ on $\mathbb{R}^d$
with finite second moments. The so-called
optimal transportation mapping  $T$ minimizes the functional
$$
\int | T(x) - x|^2 \ \mu(dx),
$$
where $|\, \cdot\, |$ is the standard Euclidean norm, among  all mappings sending $\mu$ to $\nu$: $\nu = \mu \circ T^{-1}$.
There exists a  unique ($\mu$-a.e.) mapping of this type. It turns out  (see  \cite{Vill})
that there exists a convex function
$\Phi$ such that $T$ has the form $T(x) = \nabla \Phi(x)$ for $\mu$ almost all $x$.

In the infinite-dimensional case the natural  norm to
be minimized  does not coincide
with the ambient norm. For instance,
it is well-known that the ``natural'' norm  on the
Wiener space  is the Cameron--Martin norm $|\, \cdot\, |_{H}$,
which is infinite almost everywhere.
Thus the natural infinite-dimensional Monge--Kantorovich problem on the Wiener space  deals with the functional
$$
\int | T(x) - x|^2_{H} \ \mu(dx)
$$
and two probability measures $\mu$ and $\nu$ which are absolutely continuous with respect to $\gamma$.
A sufficiently complete solution to the infinite-dimensional transportation problem on the Wiener space has been obtained  in
\cite{FU1} (see an alternative approach in \cite{Kol04}).
In particular, if
$$
{\mbox{Ent}}_{\gamma} g = \int g \log g \ d  \gamma < \infty
$$
 and $\nu = \gamma$, $\mu = g \cdot \gamma$, then
there exists an optimal transportation
$T(x) = x + \nabla \varphi(x)$ of $g \cdot \gamma$ to $\gamma$,
i.e., $\gamma = (g \cdot \gamma) \circ{T^{-1}}$,
where $\varphi$ is a function  possessing (in a certain sense)
the gradient $\nabla$ along the Cameron--Martin space;
if $\mu$ is equivalent to $\gamma$, then $\varphi$ is a $1$-convex potential
(see \cite{FU00}, where the optimal transportation
of this form is constructed from $\gamma$ to $\mu$ and our $T$ is its inverse).
Existence of an optimal transportation for any couple of probability measures
absolutely continuous with respect to $\gamma$ has been recently established in \cite{Cav}.

In addition, the following inequality
(called Talagrand's inequality) holds:
$$
\int g \log g \ d\gamma
\ge \frac{1}{2} \int |\nabla \varphi|^2 \ g  d\gamma.
$$
Moreover, there exists a mapping $S$ such that $T \circ S(x) = x$ for $g \cdot \gamma$-a.e. $x$
and
$S \circ T(x) =x$ for $\gamma$-a.e. $x$.
The mapping $S$ is an optimal transportation mapping  too
(it takes $\gamma$ to $g\cdot\gamma$)
and has the form $S(x) = x + \nabla \psi(x)$.

In this paper we study the change of variables formula. One can formally compute that the
following expression must hold:
\begin{equation}
\label{ek1.1}
g = {\det}_2({\mbox{I}} + D^2 \varphi) \exp \bigl( \mathcal{L} \varphi - \frac{1}{2} |\nabla \varphi|^2 \bigr),
\end{equation}
where  $D^2$ is the second derivative,
$$
\mathcal{L} \varphi(x) = \Delta \varphi(x) - \langle x, \nabla \varphi(x)\rangle =
\mbox{div}_{\gamma} (\nabla \varphi)(x)
$$
 is the  Ornstein--Uhlenbeck operator, and $\det_2$ is the
Fredholm--Carleman
determinant defined by  $\det_2(\mbox{I} + K) = \prod_{i=1}^{\infty} (1+k_i)e^{-k_i}$, where $K$
is a symmetric Hilbert--Schmidt operator with eigenvalues~$k_i$.
Note that $\det_2(\mbox{I} + K)\le 1$ if $I+K\ge 0$.
Relation  (\ref{ek1.1}) can be considered as an infinite-dimensional Monge--Amp{\`e}re  equation with an
unknown function  $\varphi$.

For the inverse mapping   $T^{-1}(x) = x + \nabla \psi(x)$
the change of variables formula takes the form
\begin{equation}
\label{ek1.2}
g({\rm I}+ \nabla \psi) {\det}_2(\mbox{I} + D^2 \psi) \exp \bigl( \mathcal{L} \psi - \frac{1}{2} |\nabla \psi|^2 \bigr)=1.
\end{equation}

It is a nontrivial problem  which  mappings  satisfy (\ref{ek1.1})
in the general case  (see \cite{Bo98}, \cite{Bdifm}, \cite{UZ}).
Identity  (\ref{ek1.2})  was obtained in \cite{FU1} under the assumption that  $g \cdot \gamma$
is uniformly log-concave,
i.e., $-D^2 \log g + \mbox{I} \ge \varepsilon \mbox{I}$, where $\varepsilon>0$.
It was shown in  \cite{BoKo2005}, \cite{BoKo2006} that under the assumptions
 $\mbox{Ent}_{\gamma} g < \infty$ and $g > c >0$
one has
$$
g = {\det}_2(\mbox{I} + D^2_{a} \varphi) \exp \bigl( \mathcal{L}_{a} \varphi - \frac{1}{2} |\nabla \varphi|^2 \bigr),
$$
where $D^2_{a} \varphi$ and $\mathcal{L}_{a} \varphi$  are the
absolutely continuous parts of  $D^2 \varphi$ and $\mathcal{L} \varphi$ respectively.
However, in \cite{BoKo2005}, \cite{BoKo2006} we were unable to prove the precise formula (\ref{ek1.1})
and identify $\mathcal{L}_a\varphi$ with $\mathcal{L}\varphi$.
The following theorem is the main result of this paper.

The Hilbert--Schmidt norm of an operator $A$ is denoted by $\|A\|_{\mathcal{HS}}$;
by definition $\| A \|_{\mathcal{HS}}^2  = \mbox{\rm{Tr}} (AA^*)$.

\begin{theorem}
Assume that
$\sqrt{g}\in W^{2,1}(\gamma)$, in particular
\begin{equation}\label{ek1.3}
{\rm{I}}_\gamma g:=\int \frac{|\nabla g|^2}{g} \ d\gamma < \infty.
\end{equation}
 Then $D^2\varphi(x)$ exists as a Hilbert--Schmidt operator $g\cdot\gamma$-a.e.,
$$
\int \|D^2 \varphi\|^2_{\mathcal{HS}} \ g  d\gamma < \infty,
$$
 $\mathcal{L} \varphi \in L^1(g \cdot \gamma)$ and  $g \cdot \gamma$-a.e. there holds
the change of variables formula
$$
g = {\det}_2 ({\rm{I}} + D^2 \varphi)
\exp \bigl( \mathcal{L} \varphi - \frac{1}{2} |\nabla \varphi|^2\bigr).
$$
\end{theorem}

The definitions of Sobolev classes and Sobolev derivatives are recalled below
as well as the meaning of $\nabla\varphi$ and $\mathcal{L}\varphi$. Under the additional assumption that
$1/g\in L^r(\gamma)$ with some $r>1$ we have more: $\varphi\in W^{2r/(1+r),2}(\gamma)$,
so that $\mathcal{L}\varphi$ exists in the sense of $W^{2r/(1+r),2}(\gamma)$.

Thus, now the infinite-dimensional change of variables formula is established under the assumptions
comparable to those in the finite-dimensional case.

We recall that $\mbox{I}_\gamma g$
is called Fisher's information of $g$ and the quantity
$$
\mbox{Ent}_{\gamma} g  :=
\int g \log g \ d\gamma
$$
is called the entropy of $g$ (with respect to the measure $\gamma$).

Finally, under some additional assumptions we show  higher differentiability of~$\varphi$.

Given an operator $A$ on a Hilbert space $H$ we set
$$
M(A):=\sup \bigl\{(Ah,h)\colon\, |h|_H\le 1\bigr\}.
$$
If $A$ is symmetric nonnegative, then $M(A)=\|A\|$ is the operator
norm of~$A$;  a~general bounded symmetric operator $A$ can be written
as $A=A^{+}-A^{-}$ with uniquely defined
nonnegative symmetric operators
$A^{+}$ and $A^{-}$ such that \mbox{$A^{+}A^{-}=A^{-}A^{+}=0$} and then
$M(A)=\|A^{+}\|$; the operator $A^{+}$ is called the nonnegative
part of a symmetric operator~$A$. Obviously, we always have
$M(A)\le \|A\|$.
Another new result of this paper is the following theorem.

\begin{theorem}\label{t1.2}
Assume that $g>0$ a.e.,
$g\in W^{2,1}(\gamma)\cap W^{1,2}(\gamma)$,
 and for some $p\in (1,2)$, letting $v:=-\log g$ and
 $\partial_{h}^2 v:=-\partial_h^2g/g+|\partial_h g|^2/g^2$, one has
\begin{equation}\label{ek1.4}
|\nabla g/g|\in L^{2p/(2-p)}(g\cdot \gamma)
\quad
\hbox{and}
\quad
M(\mbox{\rm{I}} + D^2 v)\in L^{p/(2-p)}(g \cdot \gamma).
\end{equation}
Let $\{e_i\}$ be an orthonormal basis in $H$. Then
$$
\int \Big( \sum_{i=1}^{\infty}
\| \partial_{e_i}(D^2 \varphi)\|^2_{\mathcal{HS}} \Bigr)^{p/2} \ g  d\gamma \le
  \biggl( \int M(\mbox{\rm{I}} + D^2 v)^{p/(2-p)}  \ g d\gamma \biggr)^{(2-p)/2}
\cdot (\mbox{\rm I}_{\gamma} g)^{p/2}.
$$
In addition, for $p=2$ one has
$$
\int \sum_{i=1}^{\infty}
\| \partial_{e_i}(D^2 \varphi)\|^2_{\mathcal{HS}} \ g d\gamma \le
  \| M(\mbox{\rm{I}} + D^2 v) \|_{L^{\infty}(g \cdot \gamma)}
\cdot \mbox{\rm I}_{\gamma} g.
$$
\end{theorem}

It should be noted that in many cases the first inclusion in (\ref{ek1.4}) follows from the
second one.

 Sections 2 and 3 are devoted to certain dimension free estimates,
which will be employed in Section 4 in the proof of the main result.
Since the proof relies on some technical improvements of a number of our earlier
results and estimates, we include the complete formulations of the corresponding
results with some explanations or proofs  where appropriate.

\section{Finite-dimensional estimates}

Let  $\gamma$ be the standard Gaussian measure on $\mathbb{R}^d$ and let $g \cdot \gamma$ be a probability measure
absolutely continuous with respect to~$\gamma$.
Consider the optimal transportation $T = \nabla \Phi$ of
$g \cdot \gamma$  to $\gamma$, where  $\Phi$ is the corresponding potential.
It is related to $\varphi$ above by the  equality $\nabla\Phi=I+\nabla\varphi$, i.e.
$\Phi(x)=\varphi(x)+\langle x,x\rangle /2$, where
$\langle x,y\rangle$ is the standard inner product in~$\mathbb{R}^d$.
Denote by $\| \cdot \|$ the operator norm and by $\| \cdot \|_{\mathcal{HS}}$
the Hilbert--Schmidt norm; $|\,\cdot\,|$ is the usual norm in~$\mathbb{R}^d$.

The Sobolev class $W^{2,1}(\gamma)$ with respect to the standard Gaussian measure $\gamma$
on $\mathbb{R}^d$ consists of all functions $f\in L^2(\gamma)$ such that
$f$ belongs to the usual local Sobolev class $W^{2,1}_{loc}(\mathbb{R}^d)$ and
$|\nabla f|\in L^2(\gamma)$. The partial derivative of a mapping $G$ with respect to $x_i$
(pointwise or Sobolev) will be denoted by $\partial_{x_i}G$ or by $G_{x_i}$.
Using $L^p$-norms and derivatives up to order $r$ one defines the classes $W^{p,r}(\gamma)$.

Some of the conditions employed below are naturally expressed in terms of certain
weighted Sobolev
spaces. Let us give definitions.
Let $\mu=g\cdot\gamma$ be a probability measure on $\mathbb{R}^d$ with $\sqrt{g}\in W^{2,1}_{loc}(\mathbb{R}^d)$.
We assume throughout that (\ref{ek1.3}) holds.
By the Gaussian log-Sobolev inequality  $g \log g \in L^1(\gamma)$.
Therefore, the coordinate functions $x_i$ belong to $L^2(g\cdot\gamma)$.
Let us introduce the Sobolev classes with respect to the measure $ g \cdot \gamma$.

\begin{definition}
We say that $f \in L^2(g \cdot \gamma)$ has the Sobolev derivative $f_{x_i} \in L^2(g \cdot \gamma)$
with respect to $x_i$ if, for every smooth compactly supported function  $\xi$ on $\mathbb{R}^d$, one has
\begin{equation}
\label{ek2.1}
\int \xi_{x_i} f \ g d\gamma = - \int \xi f_{x_i} \ g d \gamma + \int \xi f
\Bigl( x_i - \frac{g_{x_i}}{g}\Bigr) \ g d\gamma.
\end{equation}
\end{definition}

We observe  that the integrals in  (\ref{ek2.1})
are well-defined since $x_i - g_{x_i}/g \in L^2(g \cdot \gamma)$.

The space $G^{2,1}(g \cdot \gamma)$ consists of all
functions $f$ such that $f \in L^2(g \cdot \gamma)$ and
$$
\int |\nabla f|^2 \ g d\gamma = \sum_i \int |f_{x_i}|^2 \ g d\gamma  < \infty,
$$
where $f_{x_i}$ exists in the sense of the previous definition.
One can show (see \cite[Theorem 2.6.11]{Bdifm}) that $G^{2,1}(g\cdot\gamma)$ coincides with the completion
of $C_0^\infty(\mathbb{R}^d)$ with respect to the Sobolev norm
$\|f\|_{L^2(\mu)}+\bigl\| |\nabla f|\bigr\|_{L^2(\mu)}$; the latter class is denoted by
the symbol $W^{2,1}(g\cdot\gamma)$.

In the same way one defines the second Sobolev derivative $D^2 f$.
The Sobolev space $G^{1,2}(g \cdot \gamma)$  consists of
all functions $f$ such that
$$
\int |f| \ g  d\gamma +   \int |\nabla f| \ g d \gamma
+ \int \|D^2 f\|_{\mathcal{HS}}\ g d\gamma < \infty,
$$
where the derivatives are defined in the sense of (\ref{ek2.1}).

Most of the results  of this section are proven in  \cite{Kol2010}. For the reader's
convenience we give some proofs and explanations.

Throughout we repeatedly use the assumption that $\sqrt{g} \in W^{2,1}(\gamma)$
for a probability density $g$ with respect to $\gamma$, which means that $g$ has finite Fisher's
information and is equivalent to the inclusion $g\in W^{1,1}(\gamma)$ along with
$|\nabla g/g|\in L^2(g\cdot \gamma)$, where we set $\nabla g/g=0$ on the set $\{g=0\}$.
Indeed, the integrability of $|\nabla g|$ against $\gamma$ follows
from the inclusions $|\nabla g|/\sqrt{g}, \sqrt{g}\in L^2(\gamma)$.

\begin{theorem}\label{t2.2}
 {\rm(\cite{Kol2010})}
 Let $\mu = g \cdot \gamma$ be a probability measure on $\mathbb{R}^d$
 and $\sqrt{g} \in W^{2,1}(\gamma)$. If $g$ and $\Phi$ are smooth
 {\rm(}$g$ is twice continuously differentiable,
 $\Phi$ is four times continuously differentiable{\rm)}, then the following
identity holds{\rm:}
\begin{multline}\label{ek2.2}
 \mbox{\rm I}_{\gamma} g  =   \int \frac{|\nabla g|^2}{g} \ d\gamma
  =   2 \mbox{\rm Ent}_{\gamma} g
- 2 \int \log {\det}_2 (D^2 \Phi)\ g d\gamma \, +
\\
+
 \int \|D^2 \Phi -\mbox{\rm I} \|^2_{\mathcal{HS}} \  g d\gamma
 +
\sum_{i=1}^{d} \int \mbox{\rm Tr}
\bigl[ (D^2 \Phi)^{-1} (D^2 \Phi)_{x_i}\bigr]^2 \ g d\gamma,
\end{multline}
where
$
{\det}_2 (D^2 \Phi)
$
  is the Fredholm--Carleman determinant  of $D^2 \Phi$.
\end{theorem}

\begin{remark}\label{rem2.3}
\rm
We know that $\mbox{\rm Ent}_{\gamma} g < \infty$. In addition, by Jensen's inequality
$\mbox{\rm Ent}_{\sigma} \varrho\ge 0$ for any probability density~$\varrho$ with respect
to any probability measure~$\sigma$.
The other integrals on the  right-hand side
of (\ref{ek2.2}) are finite because all these expressions  are nonnegative
(since $D^2\Phi\ge 0$), so that
every term on the right is separately majorized by $\mbox{\rm I}_{\gamma} g$, that is,
$$
   -  \int \log {\det}_2 ( (D^2 \Phi)^2)\ g  d\gamma\le \mbox{\rm I}_{\gamma} g,
 $$
 $$
\sum_{i=1}^{d} \int \mbox{\rm Tr} \bigl[ (D^2 \Phi)^{-1} (D^2 \Phi)_{x_i} \bigr]^2
\ g  d\gamma \le \mbox{\rm I}_{\gamma} g.
$$
By looking at the term $2 \mbox{\rm Ent}_{\gamma} g$,
one can also consider (\ref{ek2.2}) as a stronger version of the Gaussian  log-Sobolev inequality.
\end{remark}

Moreover, (\ref{ek2.2}) along with some additional assumption
implies (see Section~3) that
$$
 \int \Bigl( \sum_{i=1}^{d}  \| (D^2 \Phi)_{x_i} \|^2_{\mathcal{HS}} \Bigr)^{1/2}
 \ g  d\gamma < \infty.
$$

It is easy to see that   (\ref{ek2.2})
makes sense for the infinite-dimensional optimal
transportation $T = {\rm I} + \nabla \varphi$ (where $\gamma$
is the standard Gaussian measure on $\mathbb{R}^{\infty}$,
i.e., the countable power of the standard
Gaussian measure on the real line).
In the infinite-dimensional case, considered in the last section,
(\ref{ek2.2}) takes the form
\begin{align}
\label{ek2.3}
 \mbox{\rm I}_{\gamma} g  & =   2 \mbox{\rm Ent}_{\gamma} g
- 2 \int \log {\det}_2 ({\rm I}+D^2 \varphi)\ g  d\gamma \, +
\\& \nonumber
+
 \int \|D^2 \varphi \|^2_{\mathcal{HS}} \  g  d\gamma
 +
\sum_{k=1}^{\infty} \int \mbox{\rm Tr} \bigl[ ({\rm I} + D^2 \varphi)^{-1}
(D^2\varphi)_{x_k} \bigr]^2 \ g  d\gamma.
\end{align}
However, this equality is not justified in this paper,
and we do not expect a proof of (\ref{ek2.3}) to be simple
 because of a number of difficult regularity issues (see Section~4).

Recall that in the finite-dimensional case $\Phi$ has generalized second derivatives that are
 bounded Borel measures satisfying the equality
$$
\int \xi_{x_i} \Phi_{x_j} \ dx = - \int \xi \ d \Phi_{x_i x_j}
$$
for every smooth compactly supported function $\xi$
and all $x_i,x_j$. Note that $D^2 \Phi$
is an operator-valued  measure and every $\Phi_{x_i x_i}$  is a nonnegative Borel measure. In addition, the
measure  $D^2 \Phi$ has an absolutely continuous part $D^2_{a} \Phi$ (the so-called second Alexandroff derivative).

We need the following finite-dimensional results.

\begin{proposition}
\label{p2.4}
Given two probability measures $f \cdot \gamma$ and $g \cdot \gamma$ on $\mathbb{R}^d$
with
$$
\mbox{\rm {Ent}}_{g \cdot \gamma} \Bigl( \frac{f}{g} \Bigr)
:= \int f \log \frac{f}{g} \ d\gamma < \infty
$$
and the corresponding optimal
transportation mappings $\nabla \Phi_f$ and $\nabla \Phi_g$ taking
 $f \cdot \gamma$ and $g \cdot \gamma$ to $\gamma$,
the following identity holds:
\begin{align*}
\mbox{\rm {Ent}}_{g \cdot \gamma} \Bigl( \frac{f}{g} \Bigr)  & = \int f \log \frac{f}{g} \ d\gamma \ge
\frac{1}{2} \int \bigl( \nabla \Phi_f - \nabla \Phi_g \bigr)^2 f \ d\gamma \, +
\\&
+  \int \Bigl( \mbox{\rm{Tr}} \bigl[D^2_a \Phi_g \cdot (D^2_a  \Phi_f)^{-1}\bigr]
 -d - \log \det\bigl[ D^2_a \Phi_g \cdot (D^2_a  \Phi_f)^{-1}\bigr]\Bigr)  \ f\  d\gamma.
\end{align*}
\end{proposition}
\begin{proof}
This result has been obtained in \cite{Kol04}, \cite{Kol2010}. For the reader's convenience we give the proof.
Without loss of generality one can assume that  $f \cdot \gamma$
is absolutely continuous with respect to
$g \cdot \gamma$  (otherwise the left-hand side is infinite).
By  McCann's change of variables formula    in $\mathbb{R}^d$
 (see \cite{McCann97}, \cite{Vill})
one has
$$
f(x) e^{-|x|^2/2} = \det D^2_a \Phi_f(x) \cdot e^{-|\nabla \Phi_f(x)|^2/2}
\quad \hbox{$f\cdot \gamma$-a.e.}
$$
 Consequently, letting $S:=(\nabla \Phi_f)^{-1}$, we have
$$
\log f (S(x)) =
 \frac{1}{2} \bigl( |S(x)|^2 - |x|^2 \bigr)
+ \log \det\bigl[ D^2_a \Phi_{f} \bigl(S(x)\bigr)\bigr]
\quad\hbox{$\gamma$-a.e.}
$$
Similarly, applying the change of variables formula for  $g$, we get
$$
\log g (S) = \frac{1}{2} \bigl( |S|^2 - |\nabla \Phi_g (S)|^2 \bigr)
+ \log \det D^2_a \Phi_{g} (S).
$$
Therefore, suppressing indication of $x$ as an argument,
\begin{align*}
 \log   \frac{f(S)}{g(S)}  &=
 \frac{1}{2} \bigl( |\nabla \Phi_g (S)|^2 - |x|^2 \bigr)
- \log \det \bigl[  D^2_a \Phi_{g} \cdot (D^2_a \Phi_{f})^{-1} \bigr](S)=
\\&
= \frac{1}{2} \bigl|\nabla \Phi_g (S) - x \big|^2
+ \langle x, \nabla \Phi_g (S) - x \rangle
- \log \det \bigl[  D^2_a \Phi_{g} \cdot ( D^2_a \Phi_{f})^{-1} \bigr](S).
\end{align*}
Let us integrate this equality with respect to $\gamma$.
Noting that $(\nabla \Phi_f)^{-1} = \nabla \Phi^*_f$,
where $\Phi^*_f$ is the dual function for $\Phi_f$
defined by $\Phi^*_f(x)=\sup_y (\langle x,y\rangle -\Phi_f(y))$,
we obtain from Lemma \ref{lem2.5} below that
\begin{multline*}
\int \langle x, \nabla \Phi_g (S(x)) - x \rangle \ \gamma(dx) \ge
\int \mbox{Tr} \bigl[ D^2_a \Phi_g (S)
 \cdot (D^2_a \Phi_f)^{-1} (S)\bigr] \ d\gamma - d=
\\
= \int \mbox{Tr} \bigl[ D^2_a \Phi_g  \cdot (D^2_a \Phi_f)^{-1}\bigr]  f \ d\gamma - d.
\end{multline*}
Thus, we have
\begin{multline*}
\int \log    \frac{f(S)}{g(S)}\ d\gamma
\ge
\frac{1}{2}\int  \bigl|\nabla \Phi_g (S(x)) - x \bigr|^2  \  \gamma(dx) \, +
\\
+ \int \Bigl[ \mbox{Tr} \bigl(D^2_a \Phi_g  \cdot (D^2_a \Phi_f)^{-1}\bigr)
 - d - \log \det \bigl( D^2_a \Phi_{g} \cdot (D^2_a \Phi_{f})^{-1} \bigr) \Bigr](S) \ d\gamma.
\end{multline*}
Taking into account that $(f \cdot \gamma) \circ (\nabla \Phi_{f})^{-1} = \gamma$ we arrive at the desired result.
\end{proof}

\begin{lemma}
\label{lem2.5}
Let $\varphi\colon\, A \to \mathbb{R}$ and $\psi\colon\, B \to \mathbb{R}$
 be two convex functions on bounded convex sets $A$ and $B$ in $\mathbb{R}^d$,
respectively. Assume that $\nabla \psi(B) \subset A$
and the image of  $\lambda|_{B}$ with respect to $\nabla \psi$, where $\lambda$ is Lebesgue measure,
  is absolutely continuous.
Then
$$
\mbox{\rm div} (\nabla \varphi \circ \nabla \psi)
\ge \mbox{\rm Tr}\bigl[ D^2_a \varphi (\nabla \psi) \cdot D^2_a \psi \bigr] \cdot\lambda \ge 0,
$$
where $\mbox{\rm div}$ is understood in the sense of distributions.
\end{lemma}
\begin{proof}
It suffices to verify this property for any ball $B$ whose closure is contained
in the interior of the original set. So we may assume that $\psi$ is defined
in a neighborhood of~$B$. Note that $|\nabla \varphi (\nabla\psi(x))|$ is locally
Lebesgue integrable, since $|\nabla \varphi|$ is locally bounded and the image of Lebesgue
measure under the mapping $\nabla\psi$ has a density.
Clearly, for smooth functions the first inequality becomes an identity.
Assume that only  $\varphi$ is smooth.
Let us approximate $\psi$ by smooth functions
$$
\psi_{\varepsilon}(x) = \int \psi(x-y) \varrho_{\varepsilon}(y) \ dy,
$$
where $\varrho_{\varepsilon}(x) = \varepsilon^{-d}\varrho(x/\varepsilon)$ and $\varrho$ is a smooth compactly supported probability
density.
It is well-known that $\nabla \psi_{\varepsilon} \to \nabla \psi $, $D^2  \psi_{\varepsilon} \to D^2_a\psi$
a.e. with respect to Lebesgue measure. This follows from the known fact that,
given two probability measures $\mu_1$ and $\mu_2$, a limit
$\lim\limits_{r \to 0} \mu_2(B_r(x))/\mu_1(B_r(x))$ exists $\mu_1$-a.e.
and it vanishes  $\mu_1$-a.e. if the measures are mutually singular.
This fact implies that $D^2 \psi_{sing} * \varrho_{\varepsilon}$ tends to zero a.e.
with respect to Lebesgue measure.
Hence for any function $\xi\in C_0^\infty(B)$ we have
\begin{align*}
\int \xi &  \cdot \mbox{div}(\nabla \varphi\circ \nabla \psi)  \ dx =
- \int \langle \nabla\xi,  \nabla \varphi\circ \nabla \psi \rangle \ dx =
- \lim_{\varepsilon \to 0} \int \langle \nabla\xi,
\nabla \varphi\circ \nabla \psi_{\varepsilon} \rangle \ dx=
\\&
= \lim_{\varepsilon \to 0} \int \xi \cdot  \mbox{\rm{Tr}}
\bigl[ D^2 \varphi (\nabla \psi_{\varepsilon}) \cdot D^2 \psi_{\varepsilon} \bigr] \ dx
\ge
\int \xi \cdot  \mbox{\rm{Tr}} \bigl[ D^2 \varphi(\nabla \psi) \cdot D^2_a \psi \bigr] \ dx.
\end{align*}
The last inequality follows from the Fatou theorem.

If the function $\varphi$ is not smooth, keeping   $\psi$ fixed, in the same way
one can construct an approximating sequence $\{\varphi_n\}$ for $\varphi$
 and repeat the above reasoning.
\end{proof}

\begin{remark}
{\rm
Note that
$$
\mbox{Tr} \bigl[ D^2_a \Phi_g  \cdot (D^2_a \Phi_f)^{-1}  \bigr]
- d - \log \det \bigl[ D^2_a \Phi_g  \cdot (D^2_a \Phi_f)^{-1}  \bigr] \ge 0.
$$
Indeed, if the operators $A, B$ are symmetric and nonnegative, then
$$
\mbox{Tr} AB- d - \log \det AB = \mbox{Tr} C- d - \log \det C,
$$
 where the operator $C= B^{1/2} A B^{1/2}$ is symmetric and nonnegative.
 The estimated quantity is $\sum_{i} (c_i -1 - \log c_i) \ge 0$, where
$c_i$ are the eigenvalues of  $C$.
}\end{remark}

A priori estimates for $\varphi$ can be proved with the help of the following
 inequality from \cite{Kol2010} (see Theorem 3.1 and Theorem 4.3 there).
 Let $e^{-V}$ be a probability density on $\mathbb{R}^d$. Then
  $x\mapsto e^{-V(x+e)}$ is a probability density for any vector~$e$.
Applying  Proposition \ref{p2.4}
 to the probability densities $f(x) = (2 \pi)^{d/2} \exp\bigl(-V(x) + |x|^2/2\bigr)$ and
 $g(x) = (2 \pi)^{d/2} \exp\bigl(-V(x+e) + |x|^2/2\bigr)$ with respect to the standard
 Gaussian measure $\gamma$ we arrive at the following result.

\begin{proposition}
For every  vector  $e \in \mathbb{R}^d$, letting $\nabla\Phi$ be
the optimal transportation of $\mu=e^{-V}\, dx$ to $\gamma$, we have
\begin{multline}
\label{ek2.4}
 \int \bigl( V(x+e) - V(x) \bigr)   e^{-V(x)} \ dx
  \ge
  \\
  \ge \int \bigl|  \nabla \Phi(x+e) - \nabla \Phi(x) \bigr|^2    e^{-V(x)} \ dx
      + \int \Bigl(\mbox{\rm Tr} D^2_a \Phi(x+ e)   \cdot  (D^2_a \Phi(x))^{-1}-
   \\
   -  d-  \log \bigl[ {\det} D^2_a \Phi(x+e)  \cdot ({\det} D^2_a\Phi(x))^{-1}\bigr]\Bigr)
   \ e^{-V(x)} \ dx.
\end{multline}
\end{proposition}

\begin{remark}
\rm
The quantity on the left in (\ref{ek2.4})
 is the entropy $\mbox{\rm {Ent}}_{g \cdot \gamma} (f/g)$
with $f(x) = (2 \pi)^{d/2} \exp\bigl(-V(x) + |x|^2/2\bigr)$ and
 $g(x) = (2 \pi)^{d/2} \exp\bigl(-V(x+e) + |x|^2/2\bigr)$.
 Hence it makes sense for every probability density
 $e^{-V} \, dx$. In the worst case it is infinite.
\end{remark}

\begin{remark}
\rm
By \cite[Proposition 4.1]{Kol2010}, for every probability
measure $\mu = e^{-V} \, dx$
with finite second moment, $V\in W^{2,1}_{loc}(\mathbb{R}^d)$
 and the logarithmic derivative $-V_{x_i}$ belonging to $L^2(\mu)$,  the optimal
 transportation $\nabla\Phi$ of $\mu$ to $\gamma$ has the property that
$\partial_{x_i} \Phi \in W^{2,1}(\mu)$ and
$$
\int (V_{x_i})^2 \ d\mu \ge \int | D^2 \Phi \cdot e_i |^2 \ d \mu.
$$
To explain the idea of the proof assume in addition that $V$ and $\Phi$
are twice continuously differentiable and $\|D^2 V\|$ is $\mu$-integrable.
Then  (\ref{ek2.4}) yields that
\begin{multline*}
\int \frac{ V(x+te_i)  + V(x-te_i) - 2V(x) }{t^2} \  e^{-V(x)} \ dx\ge
\\
  \ge \frac{1}{t^2} \int \bigl|  \nabla \Phi(x+te_i) - \nabla \Phi(x) \bigr|^2
   e^{-V(x)} \ dx
 + \frac{1}{t^2} \int \bigl|  \nabla \Phi(x-te_i) - \nabla \Phi(x) \bigr|^2 \
    e^{-V(x)} \ dx.
\end{multline*}
Letting $t \to 0$, we obtain
$$
\int \partial_{e_i}^2 V(x)  e^{-V(x)} \ dx\ge
\int |D^2 \Phi(x) \cdot e_i|^2  e^{-V(x)} \ dx.
$$
Integrating by parts on the left-hand side we get the desired estimate.

Note that the proof of this inequality can be completed  by using
only the incremental quotients and does not rely on the regularity theory for the
Monge-Amp{\`e}re equation (see \cite{Kol2010}).
\end{remark}

The following proposition is formally weaker than
Theorem \ref{t2.2}, but no smoothness of $\Phi$ and $g$
is assumed here.
This inequality it contains will be important below.

\begin{proposition}\label{p2.10}
Let $\sqrt{g} \in W^{2,1}(\gamma)$. Then $\Phi_{x_i} \in L^2(g \cdot \gamma)$
for each~$i$. Moreover, one has $\Phi_{x_i} \in W^{2,1}(g \cdot \gamma)$
and
$$
\mbox{\rm I}_{\gamma} g \ge
\int \| D^2 \Phi - \mbox{\rm{I}} \|^2_{\mathcal{HS}} \ g  d\gamma.
$$
\end{proposition}
\begin{proof}
It follows from the previous remark that
$D^2\Phi $ exists $g \cdot \gamma$-a.e.,
in particular, $D^2_a \Phi =  D^2\Phi $ \ $g \cdot \gamma$-a.e.,
 and
$$
 \int \|D^2 \Phi \|^2_{\mathcal{HS}} \ g  d\gamma < \infty,
 $$
moreover,
$$
\int \Bigr| \frac{\nabla g(x)}{g(x)} - x \Bigr|^2 g(x) \  \gamma(dx)
\ge
 \int \|D^2 \Phi \|^2_{\mathcal{HS}} \ g d\gamma.
$$
Note that
$$
\int \Bigr| \frac{\nabla g(x)}{g(x)} - x \Bigr|^2  g(x) \  \gamma(dx)
=
\int \frac{|\nabla g|^2}{g} \ d\gamma
- 2 \int \langle \nabla g(x), x\rangle \  \gamma(dx)+ \int |x|^2  g(x)\  \gamma(dx)
$$
and
$$
 \int \|D^2 \Phi \|^2_{\mathcal{HS}} \  g  d\gamma
 =
 \int \|D^2 \Phi - \mbox{I} \|^2_{\mathcal{HS}} \  g  d\gamma
 + 2 \int \Delta \Phi\ g  d\gamma -d,
$$
where second order derivatives of $\Phi$ are meant in the sense of the weighted class
$G^{1,2}(g\cdot\gamma)$.
The integration by parts formula yields that
$$
- 2 \int \langle \nabla g(x), x\rangle  \  \gamma(dx)
+ \int |x|^2  g(x)\  \gamma(dx)
=
2d - \int |x|^2 g(x) \ \gamma(dx).
$$
By the change of variables formula we have
$$
2 \int \Delta \Phi \  g  d\gamma -d =
2 \int \Delta \Phi \ g  d\gamma - \int |\nabla \Phi|^2 \  g d\gamma .
$$
Therefore,
\begin{multline*}
\int \frac{|\nabla g|^2}{g} \ d\gamma
\ge
\\
\ge
 \int \|D^2 \Phi -\mbox{I} \|^2_{\mathcal{HS}} \ g  d\gamma
+
 \int \bigl(|x|^2 - |\nabla \Phi(x)|^2\bigr)  g(x) \  \gamma(dx)
 + 2 \int (\Delta \Phi  - d) \ g d\gamma .
\end{multline*}
The change of variables formula yields that
$$
\log g(x) = \frac{|x|^2}{2} - \frac{|\nabla \Phi(x)|^2}{2} + \log \det D^2 \Phi(x).
$$
Applying this formula we complete the proof.
\end{proof}

\section{Higher smoothness of $\varphi$}

Let us establish a priori estimates of the third order derivative of the potential
function $\Phi$ whose gradient is the optimal transportation mapping, i.e. estimates
on second order derivatives of the optimal transportation mapping itself.
Estimates of this type
have been obtained in \cite{Kol2010} by using smooth approximations and regularity results
for the Monge--Amp{\`e}re equation.
Let $\gamma$ be the standard Gaussian measure on~$\mathbb{R}^d$.
Suppose that $g\cdot\gamma$ is a probability measure such that
$$
\sqrt{g} \in W^{2,1}(\gamma).
$$

\begin{theorem}\label{t3.1}
{\rm(\cite{Kol2010})}
Whenever $p \in [1,\infty)$ and $1 \le i \le d$,
for the optimal transportation $\nabla \Phi$ of $g\cdot\gamma$ to~$\gamma$ one has
$$
 \int |\Phi_{x_i x_i}|^{2p}  \ g d\gamma \le
 \Bigl(\frac{p+1}{2}\Bigr)^p \int |x_i + g_{x_i}/g|^{2p}   \ g  d\gamma,
$$
provided the integral on the right is finite.
If $g>0$ and $v:=-\log g$ is twice continuously differentiable, then
\begin{equation}\label{ek3.1}
\int \| D^2 \Phi\|^{2p}  \ g  d\gamma \le
\int [M( \mbox{\rm{I}} + D^2 v)]^{p}   \ g   d\gamma,
\end{equation}
provided the integral on the right is finite.
In the case $p=\infty$ one has the following Caffarelli-type estimate{\rm:}
\begin{equation}\label{ek3.2}
 \| D^2 \Phi(x)\|^{2}\le \sup_x M( \mbox{\rm{I}} + D^2 v(x))
\quad \hbox{$g\cdot \gamma$-a.e.}
\end{equation}
\end{theorem}

Let us turn to third order derivatives.
When dealing with $g\in W^{1,2}(\gamma)$, for the functions
$$v:=-\log g
$$
 defined almost
everywhere with
respect to the measure $g\cdot\gamma$,
we set
$$
v_{x_i x_j}:= - g^{-1} g_{x_i x_j}  + g^{-2}g_{x_i}g_{x_j} ,
$$
which is defined $g\cdot\gamma$-a.e. and coincides with the result
of formal differentiation of $g_{x_i}$ with respect to~$x_j$
on the set $\{g>0\}$. One can find a version
of $g$ that possesses the partial derivatives $g_{x_i x_j}$ almost everywhere
with respect to $g\cdot\gamma$, then $v_{x_i x_j}$ can be calculated
pointwise $g\cdot\gamma$-a.e.
Let us observe that if we also have $\sqrt{g} \in W^{2,1}(\gamma)$,
then the function $v_{x_i x_j}$ defined above belongs to $L^1(g\cdot\gamma)$ and serves
as the generalized derivative of the function
$v_{x_i}:=-g_{x_i}/g\in L^2(g\cdot\gamma)$, which is verified directly by
the integration by parts formula for $\gamma$, since $g_{x_i}/g \cdot g d\gamma =g_{x_i}\cdot d\gamma$.
Therefore, the assumptions that $\sqrt{g} \in W^{2,1}(\gamma)$ and
$g\in W^{1,2}(\gamma)$ used in the next theorem yield that
$v=-\log g\in W^{2,1}(g\cdot\gamma)\cap W^{1,2}(g\cdot\gamma)$, which looks more intrinsic,
but is less convenient technically.

\begin{theorem}\label{t3.2}
Let $\sqrt{g} \in W^{2,1}(\gamma)$, $g\in W^{1,2}(\gamma)$,
$g>0$ a.e. and
$$
|\nabla g/g| \in L^{2p/(2-p)}(g \cdot\gamma)
\quad \hbox{for some $p\in [1,2)$.}
$$
Assume, in addition, that
 $$
  M(I + D^2 v)\in L^{p/(2-p)}(g \cdot \gamma).
 $$
Then $\Phi$ has Sobolev derivatives  up to the third order
with respect to~$g\cdot\gamma$ and
\begin{equation}\label{ek3.3}
\int \Bigl(\sum_{i=1}^d \| (D^2 \Phi)_{x_i} \|_{\mathcal{HS}}^2\Bigr)^{p/2} \ g  d\gamma
\le
  \biggl(\int \bigl[M({\rm{I}} + D^2 v)\bigr]^{p/(2-p)}  \ g  d\gamma \biggr)^{(2-p)/2}
\cdot ({\rm I}_{\gamma} g)^{p/2}.
\end{equation}
\end{theorem}
\begin{proof}
We apply the reasoning that is standard in such estimates
and will be also employed in the proof of Theorem \ref{t4.6} below.
At the first step we assume that  $v$ and $\Phi$
are smooth. Applying Theorem \ref{t2.2}, we obtain
$$
\mbox{\rm I}_{\gamma} g \ge
\sum_{k=1}^{d} \int \mbox{\rm Tr}
\bigl[ (D^2 \Phi)^{-1} (D^2 \Phi)_{x_k} \bigr]^2 \ g  d\gamma.
$$
Next, using the relations (valid for positive  operators)
$$
\mbox{Tr} \bigl[ (D^2 \Phi)^{-1} (D^2 \Phi)_{x_i}\bigr]^2 =
\bigl\| (D^2 \Phi)^{-1/2} (D^2 \Phi)_{x_i} (D^2 \Phi)^{-1/2} \bigr\|^2_{\mathcal{HS}}
\ge \frac{\|(D^2 \Phi)_{x_i}\|^2_{\mathcal{HS}}}{\| D^2 \Phi\|^2}
$$
along with Theorem \ref{t3.1}, Remark \ref{rem2.3} and  H{\"o}lder's inequality, we obtain
\begin{align*}
\int \Big( & \sum_{i=1}^{d} \| (D^2 \Phi)_{x_i} \|^2_{\mathcal{HS}} \Bigr)^{p/2} \ g d\gamma\le
\\& \le
\int \| D^2 \Phi\|^p \Big(  \sum_{i=1}^{d}
\mbox{\rm{Tr}} \bigl[ (D^2 \Phi)^{-1} (D^2 \Phi)_{x_i}\bigr]^2
 \Bigr)^{p/2}  \ g d\gamma \le
\\&  \le
\biggl( \int \|D^2 \Phi\|^{2p/(2-p)}  \ g  d\gamma \biggr)^{(2-p)/2}
\cdot ({\rm I}_{\gamma} g)^{p/2}\le
\\& \le
  \biggl( \int \bigl[M({\rm{I}} + D^2 v)\bigr]^{p/(2-p)}  \ g  d\gamma \biggr)^{(2-p)/2}
\cdot ({\rm I}_{\gamma} g)^{p/2}.
\end{align*}

Let $g$ satisfy the assumptions of the theorem and, in addition, %%$g\ge c>0$.
$0<c\le g\le C$
with some constants $c$ and $C$.
Then $v\in W^{2,1}(\gamma)$.

Let $g_t= T_t g$, where $\{T_t\}$ is the Ornstein--Uhlenbeck semigroup, i.e.
$$
g_t(x)=T_t g(x) = \int g(e^{-t} x + \sqrt{1-e^{-2t}}\, y ) \  \gamma(dy).
$$
Then $g_t(x)$ is infinitely differentiable in $x$ for every $t>0$ and $c\le g_t\le C$. Let $v_t=-\log g_t$.
For every vector $h$ one has
$$
\partial_{h} v_t =  e^{-t} \frac{T_t(e^{-v} \partial_h v)}{T_t e^{-v}},
$$
\begin{equation}\label{ek3.4}
\partial^2_{h} v_t =  e^{-2t}\Bigl[ \frac{T_t(e^{-v} \partial^2_h v)}{T_t e^{-v}}
- \frac{T_t( |\partial_h v|^2 e^{-v})}{T_t e^{-v}} \
+ \frac{|T_t(e^{-v} \partial_h v)|^2}{(T_t e^{-v})^2} \Bigr] .
\end{equation}
Applying the inequality
\begin{equation}\label{ek3.5}
|T_t(uw)|^r \le T_t |u|^r \ \bigl|T_t |w|^{r/(r-1)}\bigr|^{r-1},
\end{equation}
we observe that
$$
|T_t(e^{-v} \partial_h v)|^2\le T_t(|\partial_h v|^2 e^{-v}) \ T_t e^{-v},
$$
hence
$$
\partial^2_{h} v_t\le e^{-2t} \frac{T_t(e^{-v} \partial^2_h v)}{T_t e^{-v}}.
$$
Letting $w_{t,h}=\max(0,1+\partial^2_{h} v_t)$,
$u_h=\max(0,1+\partial^2_{h} v)$, we obtain that
$$
w_{t,h}\le e^{-2t} \frac{T_t (u_h e^{-v})}{T_t e^{-v}},
$$
whence by (\ref{ek3.5}) with $u=u_{h}$ and $w=e^{-(1-1/r)v}$ we find that
$$
|w_{t,h}|^r \le e^{-2tr} \frac{T_t (|u_h|^r e^{-v})}{T_t e^{-v}},
$$
Therefore,
\begin{equation}\label{ek3.6}
\bigl[M(I + D^2 v_t)\bigr]^r e^{-v_t} \le e^{-2rt} T_t \bigl( \bigl[M(I + D^2 v)\bigr]^r  e^{-v}\bigr).
\end{equation}
It is known (see, e.g.,  \cite[Example 8.4.3]{Bdifm}) that for every
function $\psi\in L^1(\gamma)$, as $t\to 0$,
 one has $T_t\psi\to \psi$ in $L^1(\gamma)$ and almost everywhere.
It follows from this and (\ref{ek3.4}) that
 as $n\to \infty$, we have $g_{1/n}\to g$
and $D^2 v_{1/n}\to D^2v$ almost everywhere.
Estimate (\ref{ek3.6}) shows that the sequence
$\bigl[M({\rm I} + D^2 v_{1/n})\bigr]^r e^{-v_{1/n}}$ is uniformly integrable with respect to $\gamma$
once $\bigl[M({\rm I} + D^2 v)\bigr]^r  e^{-v}$ is $\gamma$-integrable. Therefore,
\begin{equation}\label{ek3.7}
\lim\limits_{n\to\infty}
\int \bigl[M({\rm I} + D^2 v_{1/n})\bigr]^r e^{-v_{1/n}}\ d\gamma
=\int \bigl[M({\rm I} + D^2 v)\bigr]^r  e^{-v} \ d\gamma.
\end{equation}
Note also that $|\nabla T_tg |^2=|T_t \nabla g|^2\le T_t (|\nabla g|^2/g) \ T_t g$,
whence we have
$$
\int \frac{|\nabla g_t|^2}{g_t}\ d\gamma \le \int \frac{|\nabla g|^2}{g}\ d\gamma .
$$
Since $g$ and $g_{1/n}$ are between $c$ and $C$, we have
${\rm Ent}_{g_{1/n}\cdot\gamma}(g/g_{1/n})\to 0$.
Proposition~\ref{p2.4} shows that for the optimal
transports $\nabla \Phi_n$ of $g_{1/n}\cdot\gamma$ to $\gamma$ we have
$|\nabla \Phi_n -\nabla\Phi|\to 0$ in $L^2(g\cdot\gamma)$.

Since for the functions $g_{1/n}$ estimate (\ref{ek3.3}) is true, taking into account (\ref{ek3.7}),
it suffices to verify that, for every fixed vector $h$ one has
\begin{equation}\label{ek3.8}
\int | \partial_{x_i} \partial_h^2\Phi|^{p} \ g d\gamma
\le \liminf_n \int | \partial_{x_i} \partial_h^2\Phi_n|^{p} \ g_{1/n}  d\gamma .
\end{equation}
Since the right side is finite, we may pass to a subsequence (denoted by the same indices)
such that the sequence of functions $g_{1/n}^{1/p} \partial_{x_i} \partial_h^2\Phi_n$
converges weakly in $L^p(\gamma)$ to some function~$w$. If we show
that $w=g^{1/p} \partial_{x_i} \partial_h^2\Phi$, the remaining estimate will be established.
It suffices to show that, for every $\xi\in C_0^\infty$,
 the integrals of $\xi w$ and $\xi g^{1/p} \partial_{x_i} \partial_h^2\Phi$
 with respect to $\gamma$ coincide.
We have
\begin{equation}\label{ek3.9}
\int \xi g^{1/p} \partial_{x_i} \partial_h^2\Phi \ d\gamma
=- \int \partial_{x_i}(\xi g^{1/p}) \partial_h^2\Phi \ d\gamma+
\int x_i \xi g^{1/p} \partial_h^2\Phi \ d\gamma .
\end{equation}
An analogous equality holds for $g_{1/n}$ in place of $g$.
Since $g\ge c>0$ and $\xi$ has compact support, it suffices to show that, for every
function $\eta\in C_0^\infty$, the functions $\eta \partial_h^2\Phi_n$ converge
to $\eta \partial_h^2\Phi$ weakly in $L^2(\gamma)$.
Due to (\ref{ek3.1}) it suffices to show that, for every $\xi\in C_0^\infty$,
 the integrals of $\xi \eta \partial_h^2\Phi_n$ against $\gamma$
converge to the integral of $\xi \eta \partial_h^2\Phi$. As above,
by the integration by parts formula this reduces to convergence of integrals with
$\partial_h\Phi_n$, which takes place, since $\partial_h\Phi_n\to \partial_h\Phi$
in $L^2(g\cdot\gamma)$. It should be noted that in the present case where $g\ge c>0$ we have even
convergence of $\partial_h^2\Phi_n$ to $\partial_h^2\Phi$ in $L^p(U)$ on every ball~$U$,
since the functions $|\nabla \partial_h^2\Phi_n|$ are uniformly bounded in~$L^p(U)$,
so the compact embedding of Sobolev spaces works.

Let us remove the assumption of two-sided boundedness of $v$. Suppose first that
$g$ is bounded from below, i.e., $g\ge c>0$.
Let us take a sequence of smooth convex functions $f_n$ on the real line such that
$f_n(s)=s$ if $s\ge -n$, $f_n(s)=-n-1$ if $s\le -n-1$, $0\le f_n'\le N$, $0\le f_n''\le N$,
where $N$ does not depend on $n$. Let us consider probability densities
$g_n=c_n e^{-v_n}$ with respect to $\gamma$,
where $v_n=f_n(v)$ and $c_n$ is a normalization constant.
Let $\nabla\Phi_n$ be the corresponding optimal transports of $g_n\cdot\gamma$ to $\gamma$.
As above, we have $\nabla\Phi_n\to \nabla\Phi$ in $L^2(g\cdot\gamma)$.
An analog of (\ref{ek3.8}) in this situation is similarly justified.
What we need is an analog of (\ref{ek3.7}). We have
$$
\partial_h v_n=f_n'(v)\partial_h v,
\quad
\partial_h^2 v_n=f_n'(v)\partial_h^2 v+f_n''(v)|\partial_h v|^2.
$$
Hence $M({\rm I}+D^2v_n)$ coincides with $M({\rm I}+D^2v)$ if $v>-n$, vanishes if $v<-n-1$
and is estimated by $N \cdot M({\rm I}+D^2v)+N |\partial_h v|^2$ if $-n-1\le v\le n$.
It remains to observe that the integral of $|\nabla v|^{2p/(2-p)}I_{-n-1\le v\le -n}g$ with respect
to $\gamma$ tends to zero as $n\to \infty$, since $|v|\in L^{2p/(2-p)}(g\cdot\gamma)$.
Thus, (\ref{ek3.7}) holds also in this case.

Finally, we reduce the general case to the considered case with bounded $v$.
We consider similar approximations $f_n(v)$, this time with concave functions such that
$f_n(s)=s$ if $s\le n$, $f_n(s)=n+1$ if $s\ge n+1$, $0\le f_n'\le N$, $f_n''\le 0$.
Defining $g_n$ as above, we again have convergence $\nabla\Phi_n\to \nabla \Phi$ in $L^2(g\cdot\gamma)$.
In place of (\ref{ek3.7}) we have a simple estimate
$$
\limsup\limits_{n\to\infty}
\int \bigl[M({\rm I} + D^2 v_{n})\bigr]^r e^{-v_{n}}\ d\gamma
\le \int \bigl[M({\rm I} + D^2 v)\bigr]^r  e^{-v} \ d\gamma,
$$
because we now have $\partial_h^2 v_n\le \partial_h^2 v$, so $M({\rm I} + D^2 v_{n})\le M({\rm I} + D^2 v)$,
in addition, on the set $\{n\le v\le n+1\}$ we have
$M({\rm I} + D^2 v)^r g_n\le 3 M({\rm I} + D^2 v)^r g$, which yields the indicated estimate.
Note that here the definition of $D^2v$ given before the theorem is used.
However, now (\ref{ek3.8}) is not obvious and requires justification, since
$g$ is not strictly positive, which makes some problems in (\ref{ek3.9}).
Namely, the problematic term is
$w=g^{1/p-1}\partial_{x_i}g \partial_h^2\Phi=g^{1/p}\partial_{x_i}v \partial_h^2\Phi$.
Writing this term as $g^{(2-p)/(2p)}\partial_{x_i}v \partial_h^2\Phi g^{1/2}$
and noting that we have convergence of
$g_n^{(2-p)/(2p)}\partial_{x_i}v_n$ to
$g^{(2-p)/(2p)}\partial_{x_i}v$ in $L^{2p/(2-p)}(\gamma)$ (which is readily verified),
we see that it suffices to show that we have weak convergence of
$\partial_h^2\Phi_n g_n^{1/2}$ to $\partial_h^2\Phi g^{1/2}$ in $L^{2p/(3p-2)}(\gamma)$.
Since $2p/(3p-2)\le 2$, it suffices to prove that there is weak convergence in $L^2(\gamma)$.
Using (\ref{ek3.1}) with $p/(2-p)$ in place of $p$
we obtain a uniform bound on the integrals of $g_n|\partial_h^2\Phi_n|^2$ against $\gamma$.
Therefore, as above, it remains to show that, for every $\xi\in C_0^\infty$,
 the integrals of $\xi g_n^{1/2} \partial_h^2\Phi_n$ against $\gamma$
 converge to the integral of $\xi g^{1/2} \partial_h^2\Phi$.
 Integrating by parts once again we see that it remains to get convergence
 of the term with
 $g_n^{-1/2} \partial_h g_n \partial_h\Phi_n=g_n^{-1}\partial_h g_n \partial_h\Phi_n g_n^{1/2}$
 to the respective term without the index~$n$.
 This convergence holds indeed, since
$g_n^{-1/2} \partial_h g_n\to g^{-1/2} \partial_h g$ in $L^2(\gamma)$
and the mappings $\nabla \Phi_n$ converge to $\nabla \Phi$ uniformly
on compact sets, which follows from their convergence in measure and convexity
of~$\Phi_n$ (see \cite[Section~25]{R}).
\end{proof}

It is important that (\ref{ek3.3}) does not depend on dimension.

\begin{remark}\label{rem3.3}
{\rm
Let us comment on the hypotheses of Theorem~\ref{t3.2}.
The inclusion $\sqrt{g}\in W^{2,1}(\gamma)$ yields that $g\in W^{1,1}(\gamma)$
and $|\nabla g|^2/g\in L^1(\gamma)$ (equivalently, $|\nabla g/g|\in L^2(g\cdot \gamma)$),
and the converse is also true.
The inclusion $\log g\in W^{r,1}(g\cdot \gamma)$ is equivalent to $|\nabla g/g|\in L^r(g\cdot \gamma)$;
the membership of $\log g$ in $W^{1,2}(g\cdot \gamma)$ can be written as
$\|D^2 g\|_{\mathcal{H}\mathcal{S}}\in L^1(\gamma)$, which is the inclusion $g\in W^{1,2}(\gamma)$.
Note also that $2p/(2-p)\ge 2$.
Therefore, the first set of our assumptions can be written as
\begin{equation}\label{ek3.10}
g\in W^{1,2}(\gamma), \quad
|\nabla g/g|\in L^{2p/(2-p)}(g\cdot\gamma).
\end{equation}
Writing our assumptions in this form is useful for
the infinite-dimensional case, since, as we shall see, passing to finite-dimensional projects
preserves these conditions.
}\end{remark}

\section{Infinite dimensional case}

Concerning analysis on the Wiener space the reader is referred to
 \cite{Bo98}, \cite{Bdifm}, \cite{Shig}, and \cite{UZ}.
Below we consider the standard Gaussian product measure $\gamma = \prod_{i=1}^{\infty} \gamma_i$ on $\mathbb{R}^{\infty}$
 with the Cameron--Martin space $H = l^2$ equipped with its
 standard Hilbert norm $|x| = \Bigl(\sum_{i=1}^{\infty} x^2_i\Bigr)^{1/2}$, where each $\gamma_i$ is the standard
Gaussian measure on the real line.
Let $\{e_i\}$ be the standard orthonormal basis in~$l^2$.
It is well-known (see \cite{Bo98}) that any centered Gaussian on a separable Fr\'echet (or, more generally,
a centered Radon Gaussian measure on a locally convex space) is isomorphic to
the product measure  $\gamma = \prod_{i=1}^{\infty} \gamma_i$ by means of a measurable linear
mapping that is one-to-one on a Borel linear subspace of full measure and is an isometry of the Cameron--Martin
spaces. For this reason the results obtained below hold in a more general setting, in particular,
for any separable Fr\'echet spaces.
The Sobolev class $W^{2,1}(\gamma)$ is introduced as the completion of the class of smooth cylindrical
functions with respect to the Sobolev norm
$\| f\|_{L^2(\gamma)}+\bigl\| |\nabla f|\bigr\|_{L^2(\gamma)}$,
where $\nabla f$ denotes the gradient along $H$, i.e., $\langle \nabla (x),h\rangle_H=\partial_h f(x)$.
Then the elements $f$ of the completion also obtain gradients $\nabla f$ along $H$ as mappings
in $L^2(\gamma, H)$ (the space of measurable $H$-valued square-integrable mappings)
specified
by the integration by parts formula
$$
\int  f \partial_{e_i}\xi \ d\gamma =-\int \xi \langle \nabla f, e_i\rangle \ d\gamma
+\int x_i f\xi\ d\gamma
$$
for smooth cylindrical functions $\xi$. Other equivalent characterizations are known
(see \cite{Bo98}, \cite{Bdifm}, \cite{Shig}). For example, $W^{2,1}(\gamma)$ coincides with
the space of all functions $f\in L^2(\gamma)$ possessing Sobolev gradients $\nabla f\in L^2(\gamma, H)$ satisfying
the above identity.
Similarly the second Sobolev class $W^{1,2}(\gamma)$ is introduced by using
the Hilbert--Schmidt norm on the second derivative $D^2 f$ along~$H$ and the $L^1$-norm;
for a general $f$ in $W^{1,2}(\gamma)$ the operator-valued mapping $D^2f$
can be specified by its matrix elements $\langle D^2f(x)e_i,e_i\rangle _H$
again through the integration by parts formula; more
general classes $W^{p,r}(\gamma)$ are naturally defined.

 Let us consider a probability measure $g \cdot \gamma$ with
 $\sqrt{g}\in W^{2,1}(\gamma)$. Then
$$
 \mbox{\rm I}_{\gamma} g=\int \frac{|\nabla g|^2}{g} \ d\gamma  < \infty.
$$
By the log-Sobolev inequality
$$
0\le 2 \mbox{\rm{Ent}}_{\gamma} g \le
 \mbox{\rm I}_{\gamma} g  < \infty.
$$

Similarly to the finite-dimensional case and the case of $W^{2,1}(\gamma)$ explained above,
one introduces the differentiation of functions
in the Sobolev sense with respect to the measure
$g\cdot \gamma$ (see \cite{Bdifm} for more details).
Namely, if $f\in L^2(g\cdot \gamma)$, then its Sobolev partial derivative
$f_{x_i}$ with respect to the variable $x_i$ is a function
in $L^1(g\cdot \gamma)$  satisfying the equality
\begin{equation}
\label{ek4.1}
\int f_{x_i} \xi \ g d\gamma = - \int f \xi_{x_i} \ g d\gamma
- \int f \xi \ \frac{g_{x_i}}{g} \ g  d\gamma
+ \int x_i f \xi \ g  d\gamma
\end{equation}
for every smooth cylindrical function  of the form
 $\xi(x) = u(x_1, \ldots, x_n)$, where $u$ is a smooth compactly supported function.
 We observe that all these integrals exist, because $f, f_{x_i}, g_{x_i}/g\in L^2(g\cdot \gamma)$
and $x_i\in L^2(g\cdot \gamma)$; the latter follows by the logarithmic Sobolev
inequality with respect to~$\gamma$. Therefore, one obtains the class $W^{2,1}(g\cdot\gamma)$
of functions $f\in L^2(g\cdot\gamma)$ such that $|\nabla f|\in L^2(g\cdot\gamma)$, where
$\nabla f=(f_{x_1},f_{x_2},\ldots)$. Similarly $W^{p,1}(g\cdot\gamma)$ and $W^{p,2}(g\cdot\gamma)$
are defined.

Let $g_n = {\rm I\!E}^{n}_{\gamma} g $ be the conditional expectation of  $g$ with respect to
$\sigma$-algebra $\mathcal{F}_n$ generated by $x_1,\ldots,x_n$ and the measure~$\gamma$.
It has the following representation:
$$
{\rm I\!E}^{n}_{\gamma} g(x_1,\ldots,x_n)
  = \int g(x_1,\ldots,x_n,y_{n+1},\ldots)  \ \Bigl( \prod_{i=n+1}^{\infty} \gamma_i \Bigr)(dy_{n+1}\cdots).
$$
 It is well-known (and follows from Jenssen's inequality) that
$$
\mbox{\rm{Ent}}_{\gamma} g_n \le \mbox{\rm{Ent}}_{\gamma} g, \ \  \mbox{\rm I}_{\gamma} g_n \le \mbox{\rm I}_{\gamma} g,
$$
hence $\sqrt{g_n} \in W^{2,1}(\gamma)$.
Since  $g_n$ depends on  finitely many coordinates, the potential
 $\varphi_n$ of the corresponding optimal transportation  $T_n(x) = x + \nabla \varphi_n(x)$ of $g_n \cdot \gamma$
 to the measure $\gamma$ depends only on the first $n$ variables.

According to Proposition \ref{p2.4} with $f=g_n$ and $g=1$ one has
$$
\mbox{\rm{Ent}}_{\gamma} g_n \ge \frac{1}{2}
\int |\nabla \varphi_n|^2 \ g_n d\gamma = \frac{1}{2} \int |\nabla \varphi_n|^2 \ g d\gamma,
$$
and Proposition \ref{p2.10} yields that
$$
\mbox{\rm I}_{\gamma} g_n  \ge
 \int \|D^2 \varphi_n \|^2_{\mathcal{HS}} \  g_n d\gamma =
 \int \|D^2 \varphi_n \|^2_{\mathcal{HS}} \  g d\gamma.
$$

In general, the gradients
$\nabla \varphi_n$ and $\nabla \varphi$ cannot be understood in the sense of  (\ref{ek4.1}),
 because in the general case no inclusions
 $\varphi_n, \varphi \notin L^2(g \cdot \gamma)$  are given.
There is no problem with functions $\varphi_n$ of finitely many variables, since their
 gradients $\nabla\varphi_n$ can be defined in the Sobolev sense locally.
Difficulties arise when we deal with~$\varphi$.
There are essentially two ways of introducing
$\nabla \varphi$ pointwise $g\cdot \gamma$-a.e.
If $g\cdot\gamma$ is equivalent to~$\gamma$,
then one can use the fact that $\varphi$ is a $1$-convex function  (see \cite{FU1});
we recall that a $\gamma$-measurable
function $f$ is called $1$-convex along the Cameron--Martin space if the function
$$
h \mapsto F_x(h):=f(x+h) + \frac{1}{2} |h|_H^2
$$
is convex on $H$ regarded as a mapping with values in the space $L^0(\gamma)$
of measurable functions with its natural ordering. In other words,
given $h,k\in H$ and $\alpha\in [0,1]$, one has
$$
F_x(\alpha h +(1-\alpha)k )\le \alpha F_x(h)+(1-\alpha)F_x(k)
\quad  \hbox{for $\gamma$-a.e.~$x$,}
$$
where the corresponding measure zero set may depend on $h,k,\alpha$.
It is also possible to consider this mapping with values in the Hilbert space $L^2(\sigma)$
for the equivalent measure $\sigma=(f^2+1)^{-1}\cdot\gamma$.
One can show that for every fixed~$i$ there is a version of $f$ such that
the functions $t\mapsto f(x+te_i)+t^2/2$ are convex. Hence almost everywhere
there exists the partial derivative $f_{x_i}$. Then we define
$\nabla f(x)$ as $(\partial_{x_i} f(x))_{i=1}^\infty$ if this element belongs to~$l^2$.

We shall define  $\nabla \varphi$ for our potential
function $\varphi$ without referring to $1$-convexity, since
we do not assume the equivalence of measures.
We shall show in a different way that, for every fixed~$i$, the function $\varphi$ has a version that
has the partial derivative $\varphi_{x_i}$ \ $g\cdot\gamma$-a.e. and the vector
$\nabla \varphi(x)=(\varphi_{x_i}(x))_{i=1}^\infty$ is in $l^2$ \ $g\cdot\gamma$-a.e.;
this amounts to the previous approach in the case of equivalent measures
(the relation to the Sobolev sense definition
is explained below).

Nevertheless, the second derivative $D^2 \varphi$
will be defined in the  Sobolev sense,
because, as we shall see,  $\varphi_{x_i} \in L^2(g \cdot \gamma)$. More precisely,
the Sobolev derivative $\varphi_{x_i x_j}  \in L^1(g \cdot \gamma)$
of $\varphi_{x_j}$ will be defined by means of~(\ref{ek4.1}).

By the finite-dimensional results we have
$$
\sup_n \int \Bigl(|\nabla \varphi_n|^2
+ \| D^2 \varphi_n \|^2_{\mathcal{HS}}\Bigr) \ g  d\gamma  < \infty.
$$
Hence, passing to  a subsequence, one can assume that  the mappings
$\nabla \varphi_n$ and $D^2 \varphi_n$
converge weakly
in the Hilbert spaces $L^2(g \cdot \gamma, H)$ and $\mathcal{H}^2_g$ defined as follows:
the space $L^2(g \cdot \gamma, H)$ is the space of measurable
mappings  $u\colon\, \mathbb{R}^{\infty} \to l^2$ with $| u|\in L^2(g\cdot \gamma)$ and
$\mathcal{H}^2_g$ is the space of measurable mappings $A$ with values in the space of symmetric Hilbert--Schmidt
operators such that $\|A \|_{\mathcal{HS}} \in L^2(g\cdot \gamma)$.

The following important result is proved in \cite{FU1} (see, in particular, Section~4 there):

{\it
one has  $\varphi_n \to \varphi$ in  $L^1(g \cdot \gamma)$
and a sequence of certain convex combinations of $\nabla \varphi_n$ converges
in $L^2(g\cdot\gamma,H)$ to a mapping denoted by $\nabla\varphi$ and having the
property that ${\rm I}+\nabla\varphi$ is the
optimal transportation taking $g \cdot \gamma$ to $\gamma$.}

However, this definition of $\nabla\varphi$
is not in the Sobolev sense.

We are going to obtain $\nabla \varphi$ similarly
as a limit of a subsequence of $\nabla\varphi_n$.
Then we would like to identify $\varphi_{x_i}$ with pointwise partial derivatives
of suitable versions  by using the integration by parts formula.
This requires some precautions since we do not know that $\varphi \in L^2(g\cdot\gamma)$,
without which
we have no inclusions $x_i\varphi, \varphi g_{x_i}/g\in L^1(g\cdot\gamma)$
and cannot refer to (\ref{ek4.1}). However, the functions
$\varphi^N = \varphi \wedge N \vee (-N)$ are bounded and,
as we shall now see, $g\cdot\gamma$-a.e. possess
partial derivatives $\varphi^N_{x_i}$ such that $|\nabla \varphi^N(x)|\le |\nabla\varphi(x)|$
\ $g\cdot\gamma$-a.e., where $\nabla \varphi^N=(\varphi^N_{x_1},\varphi^N_{x_2},\ldots)$.
In addition,
$\nabla \varphi^N(x)=\nabla\varphi(x)$ \ $g\cdot\gamma$-a.e.
on the set $\{x\colon\, |\varphi(x)|<N\}$.

\begin{proposition}
There is a  subsequence $\{n_k\}$ such that $\{\nabla \varphi_{n_k}\}$
 converges weakly in the space $L^2(g \cdot \gamma, H)$
 to some mapping denoted by $\nabla \varphi$.
 Moreover, for every $i$ there is a version of $\varphi$ denoted by the same symbol
 and possessing $g\cdot\gamma$-a.e. the partial derivative
$\varphi_{x_i}$ that coincides $g\cdot\gamma$-a.e. with $\langle \nabla\varphi, e_i\rangle_H$.

In addition, there exists $D^2 \varphi$ understood in the sense of
{\rm(\ref{ek4.1})}
and $D^2 \varphi_{n_k} \to D^2 \varphi$ weakly in $\mathcal{H}^2_g$.
In particular,
$$
\int \|D^2 \varphi \|_{\mathcal{HS}}^2 \ g  d\gamma < \infty.
$$
\end{proposition}
\begin{proof}
We only consider  convergence  of $\nabla \varphi_n$, since the case of second derivatives is
simpler.
Let us  consider the sequence
 $f_n = \varphi_n \wedge N \vee (-N)$, where $N>0$ is chosen in such a way that
 $g \cdot \gamma (\{x\colon\, \varphi(x) = \pm N\}) =0$.
 Let $f= \varphi \wedge N \vee (-N)$. Passing to a subsequence we may assume that
   $\partial_{x_i} \varphi_n \to h_i\in L^2(g\cdot\gamma)$ weakly. Then for
   all smooth cylindrical functions $\xi$ we have
\begin{align*}
\int f \xi_{x_i} \ g d\gamma   & =
\lim_{n \to \infty}
\int f_n  \xi_{x_i} \ g d\gamma=
\\
&
=  \lim_{n \to \infty} \int f_n \xi \Bigl( x_i - \frac{g_{x_i}}{g} \Bigr) \ g  d\gamma
- \lim_{n \to \infty} \int (f_n)_{x_i} \xi \ g d\gamma=
\\&
=  \int f \xi \Bigl( x_i - \frac{g_{x_i}}{g} \Bigr) \ g  d\gamma
- \lim_{n \to \infty} \int (\varphi_n)_{x_i}  I_{\{|\varphi_n| \le N\}}\xi \ g d\gamma=
\\&
=
\int f \xi \Bigl( x_i - \frac{g_{x_i}}{g} \Bigr) \ g d\gamma
-  \int h_i  I_{\{|\varphi| \le N\}}\xi \ g  d\gamma.
\end{align*}
Therefore, the Sobolev derivative of $f$ with respect to $x_i$ (and the measure $g\cdot\gamma$)
coincides with $h_i I_{\{|\varphi| \le N\}}$. In the language of differentiable
measures (see \cite{Bdifm})
this means the differentiability of the measure
$fg\cdot \gamma$ along the vector $e_i$ of the standard basis in $l^2$, which implies
(see \cite[Section~3.5 and Section~6.3]{Bdifm})
that $fg$ has a version such that the functions $t\mapsto f(x+te_i)g(x+te_i)$
are locally absolutely continuous
(that is absolutely continuous on bounded intervals)
for $\gamma$-a.e.~$x$. By our assumption,
the same is true for $\sqrt{g}$. Moreover, the derivative of $f(x+te_i)g(x+te_i)$
at $t=0$ equals $fg_{x_i}+h_i I_{\{|\varphi| \le N\}}g$ \ $\gamma$-a.e.
Once we choose a version of $g$ such that the functions $t\mapsto g(x+te_i)$
are locally absolutely continuous, we obtain a version of $f$ such
that $t\mapsto f(x+te_i)$ is absolutely continuous on every closed interval
on which the function $t\mapsto g(x+te_i)$ does not vanish.
For this version we have the estimate $|f_{x_i}|\le |h_i|$ \ $g\cdot\gamma$-a.e.
Since the conditional measures for $\gamma$ on the straight lines $x+\mathbb{R}^1e_i$
have Gaussian densities $\varrho_x$ and $h_i\in L^2(g\cdot\gamma)$, we
see that, for $\gamma$-a.e.~$x$, the integral of
$|\partial_t f(x+te_i)/\partial t|^2g(x+te_i)\varrho_x(t)$ over $\mathbb{R}$
is majorized by the integral of $|h_i(x+te_i)|^2g(x+te_i)\varrho_x(t)$.
Remembering that $f= \varphi \wedge N \vee (-N)$ depends also on $N$
suppressed in our notation and that these functions converge to $\varphi$ pointwise,
we obtain a version of $\varphi$ such that the
function $t\mapsto \varphi(x+te_i)$ is absolutely continuous on closed intervals
without zeros of~$g$.
It also follows that $h_i=\varphi_{x_i}$
almost everywhere with respect to the measure $g \cdot \gamma$.
The assertion with $\varphi_{x_ix_j}$ is even simpler, since we have
$\varphi_{x_i}\in L^2(g\cdot\gamma)$.
\end{proof}

It is important that the mapping $\nabla\varphi$ introduced in this proposition
coincides $g\cdot\gamma$-a.e. with the one constructed in \cite{FU1}
(as a limit of convex combinations). It will be seen directly that
${\rm I}+\nabla\varphi$ takes $g\cdot\gamma$ to $\gamma$ as soon as we check
that $\nabla\varphi$ can be obtained as a limit of $\nabla\varphi_n$
pointwise $g\cdot\gamma$-a.e.
Let us show that in fact we have strong convergence in $L^2(g\cdot\gamma,H)$
(which gives a subsequence convergent almost everywhere)
and convergence in the Hilbert--Schmidt norm for a subsequence in $\{D^2\varphi_n\}$.

\begin{proposition}\label{p4.2}
One has $\nabla \varphi_n \to \nabla \varphi$ in $L^2(g \cdot \gamma, H)$
and
$$
\|D^2 \varphi_n - D^2 \varphi \|_{\mathcal{HS}} \to 0\quad \hbox{$g\cdot \gamma$-a.e.}
$$
\end{proposition}
\begin{proof}
Applying Proposition \ref{p2.4} to the functions  $g_n$ and $g_m$ with $m<n$
(the conditional expectations defined above), we obtain
\begin{align*}
 \int g_n  \log \frac{g_n}{g_m} \ d\gamma &=
\mbox{\rm{Ent}}_{\gamma} g_n - \mbox{\rm{Ent}}_{\gamma} g_m
 \ge
\frac{1}{2} \int \bigl( \nabla \varphi_n - \nabla \varphi_m \bigr)^2 \ g_n  d\gamma \, -
\\
&-  \int  \log {\det}_2 \bigl(  \mbox{\rm{I}}  + D^2 \varphi_m \bigr)^{-1/2}
\bigl(  \mbox{\rm{I}}  + D^2 \varphi_n \bigr) \bigl(  \mbox{\rm{I}}
+ D^2 \varphi_m \bigr)^{-1/2} \ g_n  d\gamma =
\\& =
\frac{1}{2} \int \bigl( \nabla \varphi_n - \nabla \varphi_m \bigr)^2 \ g  d\gamma \, -
\\
&-  \int  \log {\det}_2 \bigl(  \mbox{\rm{I}}
+ D^2 \varphi_m \bigr)^{-1/2} \bigl(  \mbox{\rm{I}}
+ D^2 \varphi_n \bigr) \bigl(  \mbox{\rm{I}}  + D^2 \varphi_m \bigr)^{-1/2} \ g d\gamma;
\end{align*}
we recall that
$$
\sup_n \int \| D^2 \varphi_n\|^2 \ g  d\gamma < \infty,
$$
 hence  $D^2_a \varphi_n$ in the estimates from
 Proposition \ref{p2.4} can be replaced by $D^2 \varphi_n$.
Thus we have proved that
$$
\mbox{\rm{Ent}}_{\gamma} g_n - \mbox{\rm{Ent}}_{\gamma} g_m \ge
\frac{1}{2} \int \bigl( \nabla \varphi_n - \nabla \varphi_m \bigr)^2 \ g  d\gamma.
$$
Passing to the limit $n \to \infty$, by the properties of weak convergence we obtain
$$
\mbox{\rm{Ent}}_{\gamma} g - \mbox{\rm{Ent}}_{\gamma} g_m \ge
\frac{1}{2} \int \bigl( \nabla \varphi - \nabla \varphi_m \bigr)^2 \ g d\gamma.
$$
Now the result follows by letting $m \to \infty$.

To prove the second relation we use the convexity of  $-\log \det_2$:
\begin{align*}
\frac{1}{N} & \sum_{n=m+1}^{m+N}  \mbox{\rm{Ent}}_{\gamma} g_n - \mbox{\rm{Ent}}_{\gamma} g_m \ge
\\&
\ge -
\frac{1}{N} \sum_{n=m+1}^{m+N}  \int  \log {\det}_2\bigl[
\bigl(\mbox{\rm{I}}  + D^2 \varphi_m \bigr)^{-1/2}
\bigl(\mbox{\rm{I}}  + D^2 \varphi_n \bigr)\bigl(\mbox{\rm{I}}  + D^2 \varphi_m \bigr)^{-1/2}\bigr]
\ g d\gamma
\ge
\\&
\ge -   \int  \log {\det}_2 \Bigl[(\mbox{\rm{I}}  + D^2 \varphi_m)^{-1/2}
 \frac{1}{N} \sum_{n=m+1}^{m+N} (\mbox{\rm{I}}  + D^2 \varphi_n)
  (\mbox{\rm{I}}  + D^2 \varphi_m)^{-1/2} \Bigr] \ g d\gamma.
\end{align*}
Passing to a subsequence  (denoted again by $\varphi_n$) we obtain
$$
\frac{1}{N} \sum_{n=m+1}^{m+N} \bigl(  \mbox{\rm{I}}  + D^2 \varphi_n \bigr) \to \mbox{\rm{I}}  + D^2 \varphi
$$
in the Hilbert--Schmidt norm  $g \cdot \gamma$-a.e. Hence, by the Fatou theorem
$$
  \mbox{\rm{Ent}}_{\gamma} g- \mbox{\rm{Ent}}_{\gamma} g_m \ge
\int  \log {\det}_2\Bigl[(\mbox{\rm{I}}  + D^2 \varphi_m)^{-1/2} (\mbox{\rm{I}}  +
D^2 \varphi) (\mbox{\rm{I}}  + D^2 \varphi_m)^{-1/2} \Bigr] \ g d\gamma.
$$
Therefore, passing to a subsequence, we have
$$
\log {\det}_2\Bigl[(\mbox{\rm{I}}+ D^2 \varphi_m)^{-1/2} (\mbox{\rm{I}}
+ D^2 \varphi)(\mbox{\rm{I}}  + D^2 \varphi_m)^{-1/2}\Bigr] \to 0
\quad\hbox{$g\cdot\gamma $-a.e. as $m \to \infty$.}
$$
 Consequently,
$
D^2 \varphi_m  \to D^2 \varphi
$
in the Hilbert--Schmidt norm $g\cdot\gamma $-a.e.
\end{proof}

The next result follows from the previous proposition and the uniform  boundedness of the integrals
$\displaystyle\int \| D^2 \varphi_n \|^2_{\mathcal{HS}} \  g d\gamma$.

\begin{corollary}
In the situation of the previous proposition
$$
\int \| D^2 \varphi_n - D^2 \varphi \|^{p}_{\mathcal{HS}} \  g  d\gamma \to 0
\quad\hbox{whenever $0<p<2$.}
$$
\end{corollary}

We now prove the change of variables formula
$$
g = \det{}_{2} ({\rm I} + D^2 \varphi) \exp \Bigl( \mathcal{L} \varphi - \frac{1}{2} |\nabla \varphi|^2\Bigr),
$$
where  $\mathcal{L} \varphi$
is defined as a function in $L^1(g\cdot\gamma)$
satisfying  the following duality relation:
\begin{equation}\label{ek4.2}
\int \mathcal{L}  \varphi \ \xi \ g  d\gamma = - \int \langle \nabla \varphi,
\nabla \xi \rangle \ g  d\gamma - \int \langle \nabla g, \nabla \varphi \rangle \xi \ d\gamma
\end{equation}
for smooth cylindrical functions $\xi$; existence of $\mathcal{L} \varphi$ is also part of the proof.

\begin{lemma} The sequence
$\{\mathcal{L} \varphi_n\}$ converges $g \cdot \gamma$-a.e. to some function~$F$
and, moreover, the following change of variables formula holds:
$$
g = \det ({\rm I} + D^2 \varphi) \exp \Bigl( F - \frac{1}{2} |\nabla \varphi|^2\Bigr).
$$
\end{lemma}
\begin{proof} It follows from the finite-dimensional change of variables formula
 (\ref{ek1.1}) and convergence  $g_n \to g$, $|\nabla \varphi_n - \nabla \varphi| \to 0$,
 $\|D^2 \varphi_n - D^2 \varphi\|^2_{\mathcal{HS}} \to 0$
 that the functions  $\mathcal{L} \varphi_n(x)$ have a limit  $F(x)$ for $g \cdot \gamma$-a.e.~$x$. Clearly,
the desired formula holds in the limit.
\end{proof}

For notational simplicity, from now on we assume that the above properties established for certain subsequences hold
for the whole sequence of indices.

\begin{remark}
{\rm
We shall see that $\mathcal{L} \varphi$ coincides with
$\lim\limits_{n\to\infty}\mathcal{L} \varphi_n$ in $L^1(g\cdot\gamma)$.
It is not difficult to check that
$\{\mathcal{L} \varphi_n\}$ is bounded in $L^1(g\cdot\gamma)$, so
$\lim\limits_{n\to\infty}\mathcal{L} \varphi_n\in L^1(g\cdot\gamma)$.
Indeed, by the finite-dimensional change of variables formula  from \cite{McCann97}, \cite{Vill}
 we have
$$
g_n = {\det}_{2} ({\rm I} + D^2 \varphi_n) \exp\Bigl( \mathcal{ L} \varphi_n
- \frac{1}{2} |\nabla \varphi_n|^2 \Bigr).
$$
Hence
$$
\mathcal{L} \varphi_n = \log g_n + \frac{1}{2} |\nabla \varphi_n|^2 - \log  {\det}_{2}({\rm I} + D^2 \varphi_n).
$$
Integrating with respect to  $g_n \cdot \gamma$ and integrating in the left-hand side by parts we have
$$
\int \mathcal{L} \varphi_n \ g_n  d\gamma = -  \int \langle \nabla \varphi_n, \nabla g_n \rangle \ d\gamma.
$$
Hence
\begin{align*}
\frac{1}{2} \mbox{I}_{\gamma} g_n +  & \frac{1}{2}\int |\nabla \varphi_n|^2 \ g_n  d\gamma
\ge  -\int \langle \nabla \varphi_n, \nabla g_n \rangle \ d\gamma=
\\& = \mbox{Ent}_{\gamma} g_n + \frac{1}{2}\int |\nabla \varphi_n|^2 \ g_n d\gamma
- \int \log  {\det}_{2}({\rm I} + D^2 \varphi_n) \ g_n  d\gamma.
\end{align*}
 We see that the integrals of  $ - \log {\det}_{2}({\rm I} + D^2 \varphi_n)$
with respect to $g_n \cdot \gamma$ are finite and uniformly bounded in  $n$.
One can easily show that
$$
\sup_n \int | \mathcal{L} \varphi_n  | \ g_n  d\gamma < \infty.
$$
Indeed,
$$
|\mathcal{L} \varphi_n | \le |\log g_n |+ \frac{1}{2} |\nabla \varphi_n|^2
- \log  {\det}_{2}({\rm I} + D^2 \varphi_n).
$$
The terms on the right are nonnegative and the corresponding integrals
with respect to  $g \cdot \gamma$  are uniformly bounded in $n$.
However, convergence in $L^1(g\cdot\gamma)$ is more difficult and will be
the main step in the proof of Theorem~\ref{t4.6}.
Under the additional assumption of $\gamma$-integrability of $1/g$
to a  power greater than~$1$ we show in the final remark that
$\varphi\in W^{p,2}(\gamma)$ with some $p>1$ and $\mathcal{L}\varphi$
exists in the usual sense of functions in $W^{p,2}(\gamma)$.
}\end{remark}

It remains to identify $F$ with $\mathcal{L}\varphi$, i.e. to show that $\mathcal{L}\varphi=F$ satisfies
(\ref{ek4.2}).

\begin{theorem}
\label{t4.6}
The change of variables formula
$$
g = \det ({\rm I} + D^2 \varphi) \exp \Bigl( \mathcal{L} \varphi - \frac{1}{2} |\nabla \varphi|^2\Bigr)
$$
holds $g \cdot \gamma$-a.e.
\end{theorem}
\begin{proof}
Let us identify $F$ and $\mathcal{L} \varphi$.
One way of doing this would be proving that the integrals
of $(\mathcal{L} \varphi_n )^2$ with respect to $g\cdot \gamma$ are uniformly bounded
 and then use the uniform integrability.
However, it seems that the sequence $\{\mathcal{L} \varphi_n\}$
may be unbounded in $L^2(g\cdot\gamma)$ under the solely assumption of the
finiteness of ${\rm I}_{\gamma} g$.

To bypass this difficulty we prove another estimate:
\begin{equation}\label{ek4.3}
\sup_n  \int \frac{(\mathcal{L} \varphi_n)^2 }{1 + |\nabla \varphi_n|^2}
\ g  d\gamma \le M< \infty ,
\end{equation}
where
$$
M=4 {\rm I}_\gamma g
+2 \sup_n \int |\nabla \varphi_n|^2 \ g  d\gamma+
10\sup_n \int \|D^2\varphi_n\|_{{\mathcal HS}}^2 \ g d\gamma
\le 16  {\rm I}_\gamma g.
$$
Let  $u$ be a decreasing function on $[0,+\infty)$.
We have
\begin{align*}
& \int (\mathcal{L} \varphi_n)^2  u(|\nabla \varphi_n|^2) \ g  d\gamma =
- \sum_{e_i} \int \partial_{x_i} \varphi_n  \cdot \partial_{x_i} ( \mathcal{L} \varphi_n )
 u(|\nabla \varphi_n|^2 ) \ g  d\gamma \, -
\\& - \int \langle \nabla \varphi_n, \nabla g \rangle u(|\nabla \varphi_n|^2 )   \mathcal{L} \varphi_n \ d\gamma
- 2 \int \langle \nabla \varphi_n,
D^2 \varphi_n \nabla \varphi_n \rangle u'(|\nabla \varphi_n|^2)
  \mathcal{L} \varphi_n \ g  d\gamma.
\end{align*}
Using the relations
\begin{align*}
  - \int & \partial_{x_i} \varphi_n   \cdot \partial_{x_i}
   ( \mathcal{L} \varphi_n ) u(|\nabla \varphi_n|^2 )
 \  g  d\gamma =
 \\
&= \int (\partial_{x_i} \varphi_n)^2  u(|\nabla \varphi_n|^2 ) \  g  d\gamma
 - \int \partial_{x_i} \varphi_n  \cdot   \mathcal{L}
 ( \partial_{x_i} \varphi_n ) u(|\nabla \varphi_n|^2 )  \  g  d\gamma =
\\&
= \int (\partial_{x_i} \varphi_n)^2 u(|\nabla \varphi_n|^2 ) \  g  d\gamma
+
\int |\nabla \partial_{x_i} \varphi_n|^2  u(|\nabla \varphi_n|^2 )  \  g  d\gamma \, +
\\& + \int \partial_{x_i} \varphi_n  \cdot
\langle \nabla \partial_{x_i} \varphi_n, \nabla g \rangle
 u(|\nabla \varphi_n|^2 )  \ d\gamma \, +
\\&
+ 2 \int \partial_{x_i} \varphi_n \langle \nabla \partial_{x_i}
\varphi_n, D^2 \varphi_n \nabla \varphi_n \rangle
u'(|\nabla \varphi_n|^2 ) \  g  d\gamma
\end{align*}
and summing in  $i$ we obtain that
\begin{align*}
& \int (\mathcal{L} \varphi_n)^2  u(|\nabla \varphi_n|^2 )  \ g  d\gamma =
- \int \langle \nabla \varphi_n, \nabla g \rangle u(|\nabla \varphi_n|^2 )
\mathcal{L} \varphi_n \ d\gamma \, +
\\& - 2 \int \langle \nabla \varphi_n, D^2 \varphi_n \nabla \varphi_n \rangle
u'(|\nabla \varphi_n|^2 )    \mathcal{L} \varphi_n \ g  d\gamma \, +
\\&
+\int |\nabla \varphi_n|^2 u(|\nabla \varphi_n|^2 )  \  g  d\gamma
+
\int \|D^2 \varphi_n\|^2_{\mathcal{HS}}  u(|\nabla \varphi_n|^2 )  \  g  d\gamma \, +
\\&
+  \int \langle D^2 \varphi_n \cdot \nabla \varphi_n, \nabla g \rangle
u(|\nabla \varphi_n|^2 )  \ d\gamma
+
2 \int  \bigl| D^2 \varphi_n \nabla \varphi_n   \bigr|^2
u'(|\nabla \varphi_n|^2 )  \  g  d\gamma.
\end{align*}
For any $\varepsilon>0$ the Cauchy inequality yields
\begin{multline*}
- \int \langle \nabla \varphi_n, \nabla g \rangle u(|\nabla \varphi_n|^2 )
 \mathcal{L} \varphi_n \ d\gamma
\le
\\
\le
\frac{1}{4\varepsilon} \int \frac{|\nabla g|^2}{g} \ d\gamma
+ \varepsilon \int (\mathcal{L} \varphi_n)^2
|\nabla \varphi_n|^2 u^2(|\nabla \varphi_n|^2 ) \ g  d\gamma,
\end{multline*}
\begin{multline*}
 - 2 \int \langle \nabla \varphi_n, D^2 \varphi_n \nabla \varphi_n \rangle
u'(|\nabla \varphi_n|^2 )    \mathcal{L} \varphi_n \ g  d\gamma \le
\\
\le
\varepsilon \int   (\mathcal{L} \varphi_n)^2   u(|\nabla \varphi_n|^2)  \ g  d\gamma
+ \frac{1}{\varepsilon} \int \frac{(u')^2}{u} (|\nabla \varphi_n|^2)
 |D^2 \varphi_n \cdot \nabla \varphi_n |^2 |\nabla \varphi_n|^2   \ g d\gamma,
\end{multline*}
\begin{multline*}
 \int \langle D^2 \varphi_n \cdot \nabla \varphi_n, \nabla g \rangle
 u(|\nabla \varphi_n|^2)  \ d\gamma
\le
\\
\le
\frac{1}{4\varepsilon} \int \frac{|\nabla g|^2}{g} \ d\gamma
+\varepsilon \int | D^2 \varphi_n \cdot \nabla \varphi_n|^2
u^2(|\nabla \varphi_n|^2)  \ g  d\gamma .
\end{multline*}
It follows that
$$
\sup_n  \int (\mathcal{L} \varphi_n)^2  u(|\nabla \varphi_n|^2 ) \ g  d\gamma
< \infty,
$$
provided that the functions
$$
|\nabla \varphi_n|^2 u(|\nabla \varphi_n|^2 ), \
\frac{(u')^2}{u} (|\nabla \varphi_n|^2 )  |\nabla \varphi_n|^4
$$
are bounded and
$$
 \varepsilon u^2(|\nabla \varphi_n|^2 )  +  2  u'(|\nabla \varphi_n|^2 )\le 0.
$$
For example, we can take $u(t) = \frac{1}{1+t}$. Then both functions are bounded by~$1$
and the latter estimate holds if $\varepsilon<1$, so for $\varepsilon=1/4$
we arrive at~(\ref{ek4.3}).

The estimate obtained enables us to verify the uniform integrability
of $\{\mathcal{L}\varphi_n\}$ with respect to $g\cdot\gamma$.
Indeed, since $\{\nabla\varphi_n\}$ converges in~$L^2(g\cdot\gamma,H)$,
we have convergence of the sequence
$\{|\nabla\varphi_n|^2\}$ in $L^1(g\cdot\gamma)$, hence its uniform
integrability with respect to $g\cdot\gamma$. Now, given $\varepsilon>0$,
we can find $\delta>0$ such that the integral of $|\nabla\varphi_n|^2I_E$ against
$g\cdot\gamma$ is less than $\varepsilon^2/(4M+1)$ for every set $E$ of
$g\cdot\gamma$-measure less
than~$\delta$. Therefore, the integral of $|\mathcal{L}\varphi_n|I_E$ against
$g\cdot\gamma$ does not exceed~$\varepsilon$, because either
$$
|\mathcal{L}\varphi_n|I_E\le
\frac{\varepsilon}{2M} \frac{|\mathcal{L}\varphi_n|^2}{1+|\nabla \varphi_n|^2}
$$
and the integral over the corresponding set is estimated by $\varepsilon/2$ or
in the case of the opposite inequality we have
$|\mathcal{L}\varphi_n|I_E\le 2M\varepsilon^{-1}(1+|\nabla \varphi_n|^2)I_E$
and the integral over the corresponding set also does not exceed $\varepsilon/2$.

Finally, for any smooth cylindrical function $\xi$ we have
$$
\int \mathcal{L} \varphi_n \ \xi \ g  d\gamma
= - \int \bigl\langle\nabla  \varphi_n,
\nabla \xi + \xi \frac{\nabla g}{g} \bigr \rangle \ g d\gamma,
 $$
which gives in the limit
$$
\int F \xi \ g  d\gamma =- \int \bigl\langle\nabla
\varphi, \nabla \xi + \xi \frac{\nabla g}{g} \bigr \rangle \ g d\gamma
$$
due to the established convergence, hence $F=\mathcal{L} \varphi$, i.e. (\ref{ek4.2}) holds.
\end{proof}

\begin{remark}
{\rm
We recall once again that it has not been shown that $\varphi \in L^2(g\cdot\gamma)$
(and we do not know whether this inclusion holds under our assumptions,
under which we have only $\varphi \in L^1(g\cdot\gamma)$),
consequently, the gradient $\nabla\varphi$ has been defined not in the Sobolev sense, but
pointwise almost everywhere (however, $D^2\varphi$ is defined in the Sobolev sense
and $\nabla\varphi$ coincides almost everywhere with the limit
of the mappings $\nabla (\varphi \wedge N \vee (-N)$), where
$\varphi \wedge N \vee (-N)$ are Sobolev class functions).
In order to define also $\nabla \varphi$ in the Sobolev sense, it would be enough
to have the inclusion $\varphi\in L^2(g\cdot \gamma)$. To guarantee this inclusion,
it suffices to impose the additional condition that $g\cdot \gamma$
satisfies the Poincar\'e inequality; see also \cite{BoKo} and the next remark.
}\end{remark}

\begin{remark}
{\rm
If in the above theorem we have $1/g\in L^r(\gamma)$ for some $r>1$,
then $\varphi\in W^{p,2}(\gamma)$ with $p=2r/(1+r)$
and $\mathcal{L}\varphi$ exists in the sense
of $W^{p,2}(\gamma)$, i.e. $\varphi$ belongs to the domain of generator
of the Ornstein--Uhlenbeck semigroup in $L^p(\gamma)$.
Indeed, writing
$\|D^2\varphi_n\|_{\mathcal{HS}}^p=\|D^2\varphi_n\|_{\mathcal{HS}}^pg^{-1}g$
and  applying H\"older's inequality with respect to $g\cdot\gamma$,
we obtain a uniform bound on the integrals of $\|D^2\varphi_n\|_{\mathcal{HS}}^p$
with respect to~$\gamma$ and similarly for
$|\nabla \varphi_n|^p$, which by the Poincar\'e inequality yields that
the sequence of functions $\varphi_n-c_n$, where $c_n$ is the integral of $\varphi_n$
against~$\gamma$, is bounded in $W^{p,2}(\gamma)$, whence the claim follows.
}\end{remark}

Finally,  Theorem \ref{t1.2} follows from the finite-dimensional Theorem  \ref{t3.2}.
The proof is standard and we omit it here. As usual, one takes the finite-dimensional approximations
$g_n ={\rm I\!E}^{n}_{\gamma} g = e^{-v_n}$. Let $\gamma_n$ be the projection of~$\gamma$.
Then $g_n>0$ a.e., $g_n\in W^{2,1}(\gamma_n)\cap W^{1,2}(\gamma_n)$,
and the norm of $|\nabla g_n/g_n|$ in $L^r(g_n\cdot\gamma_n)$ is estimated by the
norm of $|\nabla g/g|$ in $L^r(g\cdot\gamma)$, whenever the latter is finite.
It is easy to show that
$$
D^2 v_n \le {\rm I\!E}^{n}_{g\cdot\gamma} (D^2 v), \ (D^2 v_n)^{+} \le
{\rm I\!E}^{n}_{g\cdot\gamma} ( D^2 v)^{+}.
$$
Hence the finite-dimensional approximations satisfy the assumptions of Theorem~\ref{t3.2}
(here Remark~\ref{rem3.3} and (\ref{ek3.10}) are useful).
The result now follows by taking the limit as $n \to \infty$.

Note that analogous results can be obtained for another interesting class of transformations,
the so-called triangular transformations (see \cite{mera}, \cite{Bdifm},
\cite{BK04}, \cite{BK05a}, \cite{BKM}).
Some a~priori estimates for optimal transportations can be found in \cite{BoKo}.

This research  was carried out within ``The National Research University Higher School of Economics''
Academic Fund Program in 2012-2013, research grant No. 11-01-0175,
and supported by the RFBR projects 10-01-00518, 11-01-90421-Ukr-f-a,
11-01-12104-ofi-m, and the program SFB 701 at the University of Bielefeld.
Critical remarks of the referee have been very useful for us.

V.B.: Department of Mechanics and Mathematics,
Moscow State University, Moscow, Russia

A.K.: Department of Mathematics,
Higher School of Economics, Moscow, Russia

\end{document}